\documentclass[11pt,reqno,a4paper]{amsart}
\usepackage{amsmath,amsthm,amssymb,amsfonts,xypic,graphics,color}  
\usepackage[english]{babel}
\usepackage[mathcal]{eucal}

\usepackage{pdfsync}
\allowdisplaybreaks
\DeclareMathOperator*{\myinf}{in\vphantom{p}f}

\begin{document}

\title[]{A compactness result in approach theory with an application to the continuity approach structure}
\author{Ben Berckmoes}
\dedicatory{Dedicated to Eva Colebunders on the occasion of her 65th birthday}
\date{}

\keywords{approach theory, compactness, continuity approach structure, Prokhorov's Theorem, uniform distance, weak topology}
\thanks{The author is postdoctoral fellow at the Fund for Scientific Research of Flanders (FWO)}

\maketitle

\newtheorem{pro}{Proposition}[section]
\newtheorem{lem}[pro]{Lemma}
\newtheorem{thm}[pro]{Theorem}
\newtheorem{de}[pro]{Definition}
\newtheorem{co}[pro]{Comment}
\newtheorem{no}[pro]{Notation}
\newtheorem{vb}[pro]{Example}
\newtheorem{vbn}[pro]{Examples}
\newtheorem{gev}[pro]{Corollary}
\newtheorem{vrg}[pro]{Question}

\newtheorem{proA}{Proposition}
\newtheorem{lemA}[proA]{Lemma}
\newtheorem{thmA}[proA]{Theorem}
\newtheorem{deA}[proA]{Definition}
\newtheorem{coA}[proA]{Comment}
\newtheorem{noA}[proA]{Notation}
\newtheorem{vbA}[proA]{Example}
\newtheorem{vbnA}[proA]{Examples}
\newtheorem{gevA}[proA]{Corollary}
\newtheorem{vrgA}[proA]{Question}

\newtheorem{proB}{Proposition}
\newtheorem{lemB}[proB]{Lemma}
\newtheorem{thmB}[proB]{Theorem}
\newtheorem{deB}[proB]{Definition}
\newtheorem{coB}[proB]{Comment}
\newtheorem{noB}[proB]{Notation}
\newtheorem{vbB}[proB]{Example}
\newtheorem{vbnB}[proB]{Examples}
\newtheorem{gevB}[proB]{Corollary}
\newtheorem{vrgB}[proB]{Question}

\newcommand{\Lin}{\textrm{\upshape{Lin}}\left(\left\{\xi_{n,k}\right\}\right)}
\newcommand{\Eta}{\textrm{\upshape{H}}}
\newcommand{\Zeta}{\textrm{\upshape{Z}}}

\hyphenation{frame-work}
\hyphenation{dif-fe-rent}
\hyphenation{a-vai-la-ble}
\hyphenation{me-tric}
\hyphenation{to-po-lo-gi-cal}
\hyphenation{con-ti--nu-ous-ly}
\hyphenation{de-pen-ding}
\hyphenation{ne-gli-gi-ble}
\hyphenation{de-ri-va-tive}

\begin{abstract}
We establish a compactness result in approach theory which we apply to obtain a generalization of Prokhorov's Theorem for the continuity approach structure.
\end{abstract}

\section{Introduction}

Measures of non-compactness (\cite{BG80}) have been studied extensively in the context of approach theory (\cite{LoIn}), both on an abstract level (\cite{BL94}, \cite{BL95}) as in specific approach settings in e.g. hyperspace theory (\cite{LS000}), functional analysis (\cite{LS00}), function spaces (\cite{L04}) and probability theory (\cite{BLV11}). The presence of a vast literature on the interplay between compactness and approach theory is explained by the fact that the latter is a canonical setting which allows for a unified treatment of the classical concept of measure of non-compactness (\cite{L88}).

In this paper we contribute to the knowledge on the interplay between compactness and approach theory. In Section 2 we provide a new compactness result for a general approach space. In Section 3 we apply this result to the specific setting of the so-called continuity approach structure (\cite{BLV13}, \cite{LoIn}) to obtain a quantitative generalization of Prokhorov's Theorem. 

\section{A compactness result in approach theory}

Let $X$ be an approach space with approach system $\mathcal{A} = \left(\mathcal{A}_x\right)_{x \in X}$. We first recall some notions related to compactness in $X$. For more details the reader is referred to \cite{LoIn}.

We say that $X$ is {\em locally countably generated} iff there exists a basis $(\mathcal{B}_x)_{x \in X}$ for  $\mathcal{A}$ such that each $\mathcal{B}_{x}$ is countable.

For $x \in X$, $\phi \in \mathcal{A}_x$ and $\epsilon > 0$ we define the {\em $\phi$-ball with center $x$ and radius $\epsilon$}\index{ball!$\phi$-ball} as the set $B_{\phi}(x,\epsilon) = \left\{y \in X \mid \phi(y) < \epsilon\right\}$. More loosely, we also refer to the latter set as a ball with center $x$ or a ball with radius $\epsilon$. 

Consider a point $x \in X$, a sequence $\left(x_n\right)_n$ in $X$ and $\epsilon > 0$. We say that $\left(x_n\right)_n$ is {\em $\epsilon$-convergent} to $x$ iff each ball $B$ with center $x$ and radius $\epsilon$ contains $x_n$  for all $n$ larger than a certain $n_B$. We write $x_n \stackrel{\epsilon}{\rightarrow} x$ to indicate that $\left(x_n\right)_n$ is $\epsilon$-convergent to $x$. We define the {\em limit operator} of $\left(x_n\right)_n$ at $x$ as 
\begin{displaymath}
\lambda\left(x_n \rightarrow x\right) = \inf\left\{\alpha > 0 \mid x_n \stackrel{\alpha}{\rightarrow} x\right\}.
\end{displaymath}

We call $X$ {\em sequentially complete} iff it holds for each sequence $(x_n)_n$ in $X$ that $\inf_{x \in X} \lambda_{\mathcal{A}}(x_n \rightarrow x) = 0$ implies the existence of a point $x_0$ to which $(x_n)_n$ converges (in the topological coreflection).

Let $A \subset X$ be a set. We say that $A$ is {\em $\epsilon$-relatively sequentially compact} iff every sequence in $A$ contains a subsequence which is $\epsilon$-convergent and we define the {\em relative sequential compactness index} of $A$ as 
\begin{displaymath}
\chi_{rsc}(A) = \inf \left\{\alpha > 0 \mid \textrm{$A$ is $\alpha$-relatively sequentially compact}\right\}. 
\end{displaymath}
Notice that relatively sequentially compact sets (in the topological coreflection) have relative sequential compactness index zero, but that the converse does not necessarily hold. 

If $\left(\Phi = \left(\phi_x\right)_x\right) \in \Pi_{x \in X} \mathcal{A}_x,$
then a set $B \subset X$ is called a {\em $\Phi$-ball} iff there exist $x \in X$ and $\alpha > 0$ such that $B = B_{\phi_x}(x,\alpha)$. We call $A$ {\em$\epsilon$-relatively compact} iff it holds for each $\Phi \in \Pi_{x \in X} \mathcal{A}_x$ that $A$ can be covered with finitely many $\Phi$-balls with radius $\epsilon$ and we define the {\em relative compactness index} of $A$ as 
\begin{displaymath}
\chi_{rc}(A) = \inf \{\alpha > 0 \mid \textrm{$A$ is $\alpha$-relatively compact}\}.
\end{displaymath}  
 
We say that $X$ is {\em $\epsilon$-Lindel{\"o}f} iff it holds for each $\Phi \in \Pi_{x \in X} \mathcal{A}_x$ that $X$ can be covered with countably many $\Phi$-balls with radius $\epsilon$ and we define the {\em Lindel{\"o}f index} of $X$ as
\begin{displaymath}
\chi_{L}(X) = \inf \left\{\alpha > 0 \mid \textrm{$X$ is $\alpha$-Lindel{\"o}f}\right\}.
\end{displaymath}

Theorem \ref{thm:LocalCompactnessLinks}, the main result of this section, interconnects the above notions. For its proof we use the following well-known lemma which belongs to the heart of approach theory (\cite{LoIn}).

\begin{lem}[Lowen]\label{lem:RegQM}
Let $\mathcal{D}_{\mathcal{A}}$ be the set of quasi-metrics $d$ on $X$ with the property that $d(x,\cdot) \in \mathcal{A}_x$ for each $x \in X$. Then the assignment of collections
\begin{displaymath}
\mathcal{B}_{\mathcal{D}_{\mathcal{A},x}} = \{d(x,\cdot) \mid d \in \mathcal{D}_{\mathcal{A}}\}
\end{displaymath}  
is a basis for $\mathcal{A}$.
\end{lem}

\begin{thm}\label{thm:LocalCompactnessLinks}
Let $X$ be locally countably generated. Then, for any set $A \subset X$,  
\begin{displaymath}
\chi_{rsc}(A) \leq \chi_{rc}(A) \leq \chi_{rsc}(A) + \chi_{L}(X). 
\end{displaymath}
In particular, if $\chi_{L}(X) = 0$, then 
\begin{displaymath}
\chi_{rsc}(A) = \chi_{rc}(A). 
\end{displaymath}
If, in addition, $X$ is sequentially complete, then 
\begin{displaymath}
\textrm{$A$ is relatively sequentially compact} \Leftrightarrow \chi_{rsc}(A) = 0.
\end{displaymath}
\end{thm}

\begin{proof}
Suppose that $A$ is not $\epsilon$-relatively sequentially compact and fix $\epsilon_0 < \epsilon$. Then there exists a sequence $\left(a_n\right)_n$ in $A$ without $\epsilon$-convergent subsequence. But then, for each $x \in X$, there exists $\phi_x \in \mathcal{A}_x$ such that the ball $B_{\phi_x}\left(x,\epsilon_0\right)$ contains at most finitely many terms of $\left(a_n\right)_n$. Indeed, if this was not the case, then the fact that $X$ is locally countably generated would allow us to extract an $\epsilon$-convergent subsequence from $\left(a_n\right)_n$. Put $\Phi = \left(\phi_x\right)_x$. Now one easily sees that $A$ cannot be covered with finitely many $\Phi$-balls with radius $\epsilon_0$. We conclude that $A$ is not $\epsilon_0$-relatively compact. We have shown that 
\begin{displaymath}
\chi_{rsc}(A) \leq \chi_{rc}(A).
\end{displaymath}

Furthermore, let $X$ be $\delta$-Lindel{\"o}f and let $A$ fail to be $\epsilon$-relatively compact, with $\delta < \epsilon$, and fix $\delta < \epsilon_0 < \epsilon$. Then Lemma \ref{lem:RegQM} enables us to choose $\Phi \in \Pi_{x \in X} \mathcal{A}_x$ of the form 
\begin{displaymath}
\Phi = \left(d_x(x,\cdot)\right)_x, 
\end{displaymath}
where each $d_x$ is a quasi-metric in $\mathcal{D}_{\mathcal{A}}$, such that $X$ cannot be covered with finitely many $\Phi$-balls with radius $\epsilon_0$. However, $X$ being $\delta$-Lindel{\"o}f, there is a countable cover $\left(B_n\right)_n$ of $X$ with $\Phi$-balls with radius $\delta$, say with centers $\left(x_n\right)_n$. Now construct a sequence $(a_n)_n$ in $A$ such that, for each $n$, the $\Phi$-ball with center $x_n$ and radius $\epsilon_0$ contains at most finitely many terms of $(a_n)_n$. But then $(a_n)_n$ has no $(\epsilon_0-\delta)$-convergent subsequence. Indeed, if any subsequence $(a_{k_n})_n$ was $(\epsilon_0 - \delta)$-convergent to $x$, then we could choose $n_0$ such that $x \in B_{n_0}$, and then it is not hard to see that the $\Phi$-ball with center $x_{n_0}$ and radius $\epsilon_0$ would contain infinitely many terms of $(a_n)_n$. We conclude that $A$ is not $(\epsilon_0 - \delta)$-relatively sequentially compact. We have established that 
\begin{displaymath}
\chi_{rc}(A) \leq \chi_{rsc}(A) + \chi_{L}(X).
\end{displaymath}

Suppose,  in addition, that $X$ is sequentially complete. Let $\chi_{rsc}(A) = 0.$ Fix a sequence $\left(a_n\right)_n$ in $A$ and carry out the following construction:\\
Choose a subsequence $\left(a_{k_1(n)}\right)_n$ and a point $x_1 \in X$ such that 
\begin{displaymath}
\lambda\left(a_{k_1(n)} \rightarrow x_1\right) \leq 1. 
\end{displaymath}
Choose a further subsequence $\left(a_{k_1 \circ k_2(n)}\right)_n$ and a point $x_2 \in X$ such that 
\begin{displaymath}
\lambda\left(a_{k_1 \circ k_2(n)} \rightarrow x_2\right) \leq 1/2. 
\end{displaymath}
\begin{displaymath}
\ldots 
\end{displaymath}
Choose a further subsequence $\left(a_{k_1 \circ \cdots \circ k_m(n)}\right)_n$ and a point $x_m \in X$ such that 
\begin{displaymath}
\lambda\left(a_{k_1 \circ \cdots \circ k_m(n)} \rightarrow x_m\right) \leq 1/m. 
\end{displaymath}
\begin{displaymath}
\ldots 
\end{displaymath}
Then it holds for the diagonal subsequence 
\begin{displaymath}
\left(a_n^\prime = a_{k_1 \circ \cdots \circ k_n(n)}\right)_n 
\end{displaymath}
that  $\inf_{x \in X} \lambda\left(a_n^\prime \rightarrow x\right) = 0.$ Now the sequential completeness of $X$ allows us to conclude that $\left(a_n^\prime\right)_n$ is convergent. We infer that $A$ is relatively sequentially compact.
\end{proof}

\section{Compactness for the continuity approach structure}

 We start by recalling some basic concepts. They can be found in any standard work on probability theory (e.g. \cite{Kal}).

A {\em cumulative distribution function} (cdf) is a non-decreasing and right-continuous map $F : \mathbb{R} \rightarrow \mathbb{R}$ for which $\displaystyle{\lim_{x \rightarrow - \infty} F(x) = 0}$ and $\displaystyle{\lim_{x \rightarrow \infty} F(x) = 1}$. The collection of (continuous) cdf's is denoted as $\mathcal{F}_{(c)}$.

The {\em weak topology} $\mathcal{T}_w$ on $\mathcal{F}$ is the initial topology for the source
\begin{displaymath}
\left(\mathcal{F} \rightarrow \mathbb{R} : F \mapsto \int_{-\infty}^\infty h(x) dF(x)\right)_{h \in \mathcal{C}_b(\mathbb{R},\mathbb{R})}
\end{displaymath}
with $\mathcal{C}_b(\mathbb{R},\mathbb{R})$ the set of bounded and continuous maps $h : \mathbb{R} \rightarrow \mathbb{R}$.

The {\em uniform distance} between $F$ and $G$ in $\mathcal{F}$ is 
\begin{displaymath}
D_u(F,G) = \sup_{x \in \mathbb{R}} \left|F(x) - G(x)\right|.
\end{displaymath}
 
The {\em convolution product} of $F$ and $G$ in $\mathcal{F}$ is the cdf
\begin{displaymath}
F \star G = \int_{-\infty}^\infty F(\cdot - y) dG(y).
\end{displaymath}
This product is commutative and $F \star G \in \mathcal{F}_c$ if $F \in \mathcal{F}_c$.

The following classical result, which can be found in \cite{Berg}, shows how the previous concepts are interconnected. We denote the underlying topology of $D_u$ as $\mathcal{T}_{D_u}$.

\begin{thm}[Bergstr{\"o}m]\label{thm:Berg}
The source 
\begin{displaymath}
\left(\left(\mathcal{F},\mathcal{T}_w\right) \rightarrow \left(\mathcal{F}_c,\mathcal{T}_{D_u}\right): F \mapsto F \star G\right)_{G \in \mathcal{F}_c}
\end{displaymath}
is initial.
\end{thm}

In order to lay the structural foundations of the quantitative central limit theory developed in \cite{BLV13}, approach theory was invoked. More precisely, the following definition was proposed in \cite{BLV13} (appendix B). It is clearly inspired by Theorem \ref{thm:Berg}. We denote the underlying approach structure of $D_u$ as $\mathcal{A}_{D_u}$.

\begin{de}\label{de:ContApp}
The {\em continuity approach structure} $\mathcal{A}_c$ on $\mathcal{F}$ is the initial approach structure for the source
\begin{displaymath}
\left(\mathcal{F} \rightarrow \left(\mathcal{F}_c,\mathcal{A}_{D_u}\right) : F \mapsto F \star G\right)_{G \in \mathcal{F}_c}.
\end{displaymath}
\end{de}

Define for each cdf $F$ and each $\alpha \in \mathbb{R}^+_0$ the mapping
\begin{displaymath}
\phi_{F,\alpha} : \mathcal{F} \rightarrow \left[0,1\right] \index{$\phi_{F,\alpha}$}
\end{displaymath}
by putting
\begin{displaymath}
\phi_{F,\alpha}(G) = \sup_{x \in \mathbb{R}} \max\{F(x - \alpha) - G(x),G(x) - F(x + \alpha)\}.
\end{displaymath}
Furthermore, for $F \in \mathcal{F}$, let $\Phi(F)$ be the set of all maps $\phi_{F,\alpha}$, where $\alpha$ runs through $\mathbb{R}^+_0$.

The proofs of the following results can be found in \cite{LoIn}. 

\begin{thm}\label{thm:BasisCApp}
The collection of sets $\left(\Phi(F)\right)_{F \in \mathcal{F}}$ is a basis for the approach system of $\mathcal{A}_c$. 
\end{thm}

\begin{thm}\label{thm:RelApp}
The topological coreflection of $\mathcal{A}_c$ is $\mathcal{T}_w$. The metric coreflection of $\mathcal{A}_c$ is $D_u$.
\end{thm}

\begin{thm}\label{thm:sc}
The space $({\mathcal{F},\mathcal{A}_c})$ is locally countably generated and sequentially complete.
\end{thm}

It is the aim of this section to express the relative sequential compactness index of a set $\mathcal{D}$ in the space $(\mathcal{F},\mathcal{A}_c)$ in terms of a canonical index measuring up to what extent $\mathcal{D}$ is tight (\cite{Kal}). We thus obtain a strong quantitative generalization of Prokhorov's Theorem (Theorem \ref{Prokhorov}). To this end, we make use of the compactness result obtained in the previous section. First some preparation is required.

Define, for $\gamma \in \mathbb{R}^+_0$, the metric $L_\gamma(F,G)$ between $F$ and $G$ in $\mathcal{F}$ as the infimum of all $\alpha \in \mathbb{R}^+_0$ for which the inequalities 
\begin{displaymath}
F(x - \gamma \alpha ) -  \alpha \leq G(x) \leq F(x + \gamma \alpha) + \alpha
\end{displaymath}
hold for all points $x \in \mathbb{R}$. The metric $L_\gamma$ is known as the L{\'e}vy metric with parameter $\gamma$ (\cite{Kal}).

\begin{thm}\label{pro:ContAppUnderLevMet}
The assignment of collections
\begin{equation}
\left(\left\{ L_{\gamma}(F,\cdot) \mid \gamma \in \mathbb{R}^+_0\right\}\right)_{F \in \mathcal{F}}\label{LevyBase}
\end{equation}
is a basis for the approach system of $\mathcal{A}_c$.
\end{thm}

\begin{proof}
It is easily seen that $L_{\gamma_1} \leq L_{\gamma_2}$ whenever $\gamma_2 \leq \gamma_1$, whence (\ref{LevyBase}) is a basis for an approach structure which we denote $\mathcal{A}$. Now it is enough to prove that, for all $F \in \mathcal{F}$ and $\mathcal{D} \subset \mathcal{F}$ nonempty, $\delta_{\mathcal{A}}(F,\mathcal{D}) = \delta_{\mathcal{A}_c}(F,\mathcal{D})$. We will do this in two steps, making use of Theorem \ref{thm:BasisCApp}.

1) $\delta_{\mathcal{A}}(F,\mathcal{D}) \leq \delta_{\mathcal{A}_c}(F,\mathcal{D})$: If $\delta_{\mathcal{A}_c}(F,\mathcal{D}) < \theta$ with $\theta > 0$, then for $\gamma > 0$ there exists $G \in \mathcal{D}$ for which $\phi_{F,\gamma \theta}(G) < \theta.$ But then we have for all real numbers $x$ that $F(x - \gamma \theta) - G(x) < \theta$ and $G(x) - F(x + \gamma \theta) < \theta$, from which we deduce that $L_\gamma(F,G) < \theta$ and hence $\delta_{\mathcal{A}}(F,\mathcal{D})\leq \theta$, which proves the desired inequality.

2) $\delta_{\mathcal{A}_c}(F,\mathcal{D}) \leq \delta_{\mathcal{A}}(F,\mathcal{D})$: If $\delta_{\mathcal{A}_c}(F,\mathcal{D}) > \theta$ with $\theta > 0$, then there exists $\alpha > 0$ such that for all $G \in \mathcal{D}$ we have $\phi_{F,\alpha}(G) > \theta$. If we put $\gamma = \alpha \theta^{-1}$, then it follows that for every $G \in \mathcal{D}$ there exists $x \in \mathbb{R}$ such that $F(x - \gamma \theta) - G(x) > \theta$ or $G(x) - F(x + \gamma \theta) > \theta$. We conclude that $L_{\gamma}(F,G) \geq \theta$ and hence $\delta_{\mathcal{A}}(F,\mathcal{D}) \geq \theta$, which proves the desired inequality.
\end{proof}

\paragraph{} We call a finite set of points at which $F$ is continuous an \textit{$F$-net}\index{$F$-net} and we introduce for each $F$-net $\mathcal{N}$ the mapping 
\begin{displaymath}
\psi_{F,\mathcal{N}} : \mathcal{F} \rightarrow \left[0,1\right]\index{$\psi_{F,\mathcal{N}}$} 
\end{displaymath}
by setting
\begin{displaymath}
\psi_{F,\mathcal{N}}(G) = \sup_{x \in N} \left|F(x) - G(x)\right|.
\end{displaymath}

\begin{lem}\label{BasesCApp}
For every $F \in \mathcal{F}$ the following hold.

1) For an $F$-net $\mathcal{N}$ and $\epsilon > 0$ there exists $\alpha  \in \mathbb{R}^+_0$ so that 
\begin{displaymath}
\psi_{F,\mathcal{N}}(G) \leq \phi_{F,\alpha}(G) + \epsilon
\end{displaymath}
for each $G \in \mathcal{F}$.

2) For $\alpha \in \mathbb{R}^+_0$ and $\epsilon > 0$ there exists an $F$-net $\mathcal{N}$ so that 
\begin{displaymath}
\phi_{F,\alpha}(G) \leq \psi_{F,\mathcal{N}}(G) + \epsilon
\end{displaymath}
for each $G \in \mathcal{F}$.

\end{lem}
\begin{proof}
Let $F \in \mathcal{F}$.

1) Fix an $F$-net $\mathcal{N}$ and $\epsilon > 0$. Since all $x \in \mathcal{N}$ are continuity points of $F$, we may choose $\alpha \in \mathbb{R}^+_0$ such that 
\begin{displaymath}
\forall x \in \mathcal{N}, \forall y \in X : \left|x - y\right| \leq \alpha \Rightarrow \left|F(x) - F(y)\right| \leq \epsilon.
\end{displaymath}
Now, for $G \in \mathcal{F}$ and $x \in N$, we have on the one hand 
\begin{displaymath}
F(x) - G(x) \leq F(x - \alpha) - G(x) + \epsilon \leq \phi_{F,\alpha}(G) + \epsilon,
\end{displaymath}  
and on the other
\begin{displaymath}
G(x) - F(x) \leq G(x) - F(x + \alpha) + \epsilon \leq \phi_{F,\alpha}(G) + \epsilon,
\end{displaymath}
from which it follows that 
\begin{displaymath}
\psi_{F,\mathcal{N}}(G) \leq \phi_{F,\alpha}(G) + \epsilon
\end{displaymath}
and we are done.

2) Fix $\alpha \in \mathbb{R}^+_0$ and $\epsilon > 0$. The number of discontinuities of $F$ being at most countable, it is possible to construct an $F$-net $\mathcal{N}$ consisting of points 
\begin{displaymath}
x_0 < x_1 < \ldots < x_{n-1} < x_n 
\end{displaymath}
such that $F(x_0) \leq \epsilon$, $x_{i+1} - x_i < \alpha$ for all $i \in \{0,\ldots, n-1\}$ and $F(x_n) \geq 1 - \epsilon$. Now fix $G \in \mathcal{F}$ and $x \in \mathbb{R}$. We distinguish between the following cases.\\
If there exists $i \in \{0,\ldots,n-1\}$ such that $x_i \leq x < x_{i+1}$, then 
\begin{displaymath}
F(x-\alpha) - G(x) \leq F(x_i) - G(x_i) \leq \psi_{F,\mathcal{N}}(G) + \epsilon
\end{displaymath} 
and
\begin{displaymath}
G(x) - F(x + \alpha) \leq G(x_{i+1}) - F(x_{i+1}) \leq \psi_{F,\mathcal{N}}(G) + \epsilon. 
\end{displaymath}
If $x < x_0$, then
\begin{displaymath}
F(x - \alpha) - G(x) \leq F(x_0) \leq \epsilon \leq \psi_{F,\mathcal{N}}(G) + \epsilon
\end{displaymath}
and
\begin{displaymath}
G(x) - F(x + \alpha) \leq G(x) - F(x) \leq G(x_0) - (F(x_0) - \epsilon) \leq \psi_{F,\mathcal{N}}(G) + \epsilon.
\end{displaymath}
If $x \geq x_n$, then 
\begin{displaymath}
F(x - \alpha) - G(x) \leq F(x) - G(x) \leq (F(x_n) + \epsilon) - G(x_n) \leq \psi_{F,\mathcal{N}}(G) + \epsilon
\end{displaymath}
and
\begin{displaymath}
G(x) - F(x + \alpha) \leq \epsilon \leq \psi_{F,\mathcal{N}}(G) + \epsilon.
\end{displaymath}
Hence we conclude that 
\begin{displaymath}
\phi_{F,\alpha}(G) \leq \psi_{F,\mathcal{N}}(G) + \epsilon,
\end{displaymath}
which finishes the proof.
\end{proof}

\paragraph{} For $F \in \mathcal{F}$, let $\Psi(F)$ be the set of all maps $\psi_{F,\mathcal{N}}$, with $\mathcal{N}$ running through all $F$-nets.

\begin{thm}\label{thm:AltBaseCApp}
The collection of sets $(\Psi(F))_{F \in \mathcal{F}}$ is a basis for the approach system of $\mathcal{A}_c$.
\end{thm}

\begin{proof}
Combine Theorem \ref{thm:BasisCApp} and Lemma \ref{BasesCApp}.
\end{proof}

The following result provides us with information about the Lindel{\"o}f index of the space $(\mathcal{F},\mathcal{A}_c)$.

\begin{thm}\label{pro:ContLindZero}
We have
\begin{displaymath}
\chi_{L}\left(\mathcal{F},\mathcal{A}_c\right) = 0.
\end{displaymath}
\end{thm}

\begin{proof}
Fix a basis $\left(\mathcal{B}_F\right)_{F \in \mathcal{F}}$ for $\mathcal{A}_c$, $\displaystyle{\left(\phi_F\right)_{F \in \mathcal{F}} \in \Pi_{F \in \mathcal{F}} \mathcal{B}_F}$ and $\epsilon > 0$. The space $(\mathcal{F},\mathcal{T}_w)$ being separable (\cite{Pa}), we fix a countable set $\mathcal{D} \subset \mathcal{F}$ which is dense for the weak topology. Now, the assignment of collections 
\begin{displaymath}
(\{L_{1/n}(F,\cdot) \mid n \in \mathbb{N}_0\})_{F \in \mathcal{F}}
\end{displaymath}
being a basis for $\mathcal{A}_c$ (Theorem \ref{pro:ContAppUnderLevMet}), there exist for each $F \in \mathcal{F}$: 

1) a number $n_F \in \mathbb{N}_0$ such that 
\begin{equation}
\phi_F(G) < L_{1/n_F}(F,G) + \epsilon/3\label{eq1} 
\end{equation}
for each $G \in \mathcal{F}$

2) an element $D_F \in \mathcal{D}$ such that 
\begin{equation}
L_{1/n_F}(F,D_F) < \epsilon/3.\label{eq2} 
\end{equation}
Combining (\ref{eq1}) and (\ref{eq2}) we have, for $F,G \in \mathcal{F}$, 
\begin{eqnarray}
\phi_F(G) &<& L_{1/n_F}(F,G) + \epsilon/3\label{eq3}\\ 
&\leq& L_{1/n_F}(F,D_F) + L_{1/n_F}(D_F,G) + \epsilon/3\nonumber\\
&<& L_{1/n_F}(D_F,G) + 2\epsilon/3.\nonumber
\end{eqnarray}
Now consider the function 
\begin{displaymath}
\zeta : \mathcal{F} \rightarrow \mathbb{N}_0 \times \mathcal{D} 
\end{displaymath}
defined by $\zeta(F) = (n_F,D_F)$ and fix for each $(n,D) \in \zeta(\mathcal{F})$ a cdf $H_{n,D} \in \mathcal{F}$ such that  $\zeta(H_{n,d})  = (n,D)$. Thus, by (\ref{eq3}),
\begin{equation}
\phi_{H_{n,D}}(G) < L_{1/n}(D,G) + 2\epsilon/3\label{eq4}
\end{equation}
for each $G \in \mathcal{F}$. Consider the countable set $\mathcal{C} = \{H_{n,D} \mid (n,D) \in \zeta(\mathcal{F})\}.$ We claim that 
\begin{displaymath}
\sup_{F \in \mathcal{F}} \myinf_{C \in \mathcal{C}} \phi_{C}(H) \leq \epsilon. 
\end{displaymath}
Indeed, for $F \in \mathcal{F}$ it suffices to consider the point $H_{n_F,D_F} \in \mathcal{C}$ since by (\ref{eq1}) and (\ref{eq4})
\begin{displaymath}
\phi_{H_{n_F,D_F}}(F) < L_{1/n_F}(D_F,F) + 2\epsilon/3 < \epsilon. 
\end{displaymath}
This finishes the proof.
\end{proof}

\paragraph{} We call a set $\mathcal{D} \subset \mathcal{F}$ {\em weakly relatively sequentially compact} iff it is relatively sequentially compact under the weak topology on $\mathcal{F}$, i.e. each sequence in $\mathcal{D}$ contains a weakly convergent subsequence.

\paragraph{} Since the weak topology is the topological coreflection of $\mathcal{A}_c$ (Theorem \ref{thm:RelApp}), the space $(\mathcal{F},\mathcal{A}_c)$ is locally countably generated and sequentially complete (Theorem \ref{thm:sc}) and $\chi_L\left(\mathcal{F},\mathcal{A}_c\right) = 0$ (Theorem \ref{pro:ContLindZero}), we may apply Theorem \ref{thm:LocalCompactnessLinks} to conclude that

\begin{thm}\label{thm:LocalCompactnessLinksContApp}
For a set $\mathcal{D} \subset \mathcal{F}$ we have
\begin{displaymath}
\left(\chi_{rsc}\right)_{\mathcal{A}_c}(\mathcal{D}) = \left(\chi_{rc}\right)_{\mathcal{A}_c}(\mathcal{D}). 
\end{displaymath}
Furthermore,
\begin{displaymath}
\textrm{$\mathcal{D}$ is weakly relatively sequentially compact} \Leftrightarrow \left(\chi_{rsc}\right)_{\mathcal{A}_c}(\mathcal{D}) = 0.
\end{displaymath}
\end{thm}

\paragraph{} Recall that a collection $\mathcal{D} \subset \mathcal{F}$ is {\em tight} (\cite{Kal}) iff for each $\epsilon > 0$ there exists a constant $M  \in \mathbb{R}^+_0$ such that $\max \left\{F(-M) ,1 - F(M)\right\} \leq \epsilon$ for all $F \in \mathcal{D}$.

\paragraph{} We now define the number
\begin{eqnarray*}
\chi_e(\mathcal{D}) = \myinf_{M > 0} \sup_{F \in \mathcal{D}} \max\left\{F(-M),1 - F(M)\right\}.\index{$\chi_e(\mathcal{D})$} 
\end{eqnarray*} 
We call $\chi_e(\mathcal{D})$ the \textit{escape index of}\index{escape index} $\mathcal{D}$ (not to be confused with the tightness indices discussed in \cite{BLV11}). Notice that $\mathcal{D}$ is tight if and only if $\chi_e(\mathcal{D}) = 0$.

\paragraph{} The following simple example shows that the escape index produces meaningful non-zero values.

\begin{vb}
Fix $0 < \alpha < 1$ and let $\mathcal{F}$ be the set of all probability distributions $F_n = (1-\alpha)F_{\delta_0} + \alpha F_{\delta_n}, \phantom{1} n \in \mathbb{N}_0$, $F_{\delta_x}$ standing for the Dirac probability distribution making a jump of height 1 at $x$. Then $\chi_e(\mathcal{D}) = \alpha$.
\end{vb}

\paragraph{} We finally come to a quantitative generalization of Prokhorov's Theorem for the continuity approach structure. As in the classical case, the proof is based on Helly's Selection Principle (\cite{Kal}).

\begin{thm}[Helly's Selection Principle]
Fix a number $M \in \mathbb{R}^+_0$ and a sequence $(F_n : \left[-M,M\right[ \rightarrow \left[0,1\right])_n$ of non-decreasing right-continuous functions. Then there exists a subsequence $\left(F_{k_n}\right)_n$ and a non-decreasing right-continuous function $F : \left[-M,M\right[ \rightarrow \left[0,1\right]$ such that $F_{k_n}(x) \rightarrow F(x)$ for each point $x$ at which $F$ is continuous.
\end{thm}

\begin{thm}[Quantitative Prokhorov's Theorem]\label{Prokhorov}
For $\mathcal{D} \subset \mathcal{F}$ we have
\begin{displaymath}
\left(\chi_{rsc}\right)_{\mathcal{A}_c}(\mathcal{D}) = \chi_e(\mathcal{D}).
\end{displaymath}
\end{thm}

\begin{proof}
Recall that, by Theorem \ref{thm:LocalCompactnessLinksContApp}, 
\begin{displaymath}
\left(\chi_{rsc}\right)_{\mathcal{A}_c}\left(\mathcal{D}\right) = \left(\chi_{rc}\right)_{\mathcal{A}_c}\left(\mathcal{D}\right).
\end{displaymath}

1) $\left(\chi_{rsc}\right)_{\mathcal{A}_c}(\mathcal{D}) \leq \chi_e(\mathcal{D})$: Fix $\epsilon > 0$ and a sequence $(F_n)_n$ in $\mathcal{D}$. Now we choose a constant $M \in \mathbb{R}^+_0$ in such a way that for each $F \in \mathcal{D}$ it holds that $\max \left\{F(-M), 1 - F(M)\right\} \leq \chi_e(\mathcal{D}) + \epsilon$. Then Helly's Selection Principle furnishes a subsequence $\left(F_{k_n}\right)_n$ and a non-decreasing right-continuous function $G : \left[-M,M\right[ \rightarrow \left[0,1\right]$ such that $F_{k_n}(x) \rightarrow G(x)$ for all points $x$ at which $G$ is continuous. Finally, we define $\widetilde{G} \in \mathcal{F}$ by
\begin{displaymath}
 \widetilde{G}(x) = \left\{\begin{array}{clrrr}      
0 	& \textrm{ if }& x < -M\\       
G(x) & \textrm{ if }& -M \leq x < M\\
1 & \textrm{ if }& x \geq M
\end{array}\right..
\end{displaymath}
But then, by Theorem \ref{thm:AltBaseCApp}, we clearly have $\lambda_{\mathcal{A}_c}\left(F_{k_n} \rightarrow \widetilde{G}\right) \leq \chi_e(\mathcal{D}) + \epsilon$ and hence $\left(\chi_{rsc}\right)_{\mathcal{A}_c}(\mathcal{D}) \leq \chi_e(\mathcal{D})$.

2) $\chi_e(\mathcal{D}) \leq \left(\chi_{rc}\right)_{\mathcal{A}_c}(\mathcal{D})$: Let $\epsilon > 0$. Then for $\alpha \in \mathbb{R}^+_0$ there exists a finite collection $\mathcal{E} \subset \mathcal{F}$ such that for all $F \in \mathcal{D}$ we can find $G \in \mathcal{E}$ for which $\phi_{G,\alpha}(F) \leq \left(\chi_{rc}\right)_{\mathcal{A}_c}(\mathcal{D}) + \epsilon/2$. Since $\mathcal{E}$ is finite we may choose a constant $\widetilde{M} \in  \mathbb{R}^+_0$ such that for each $G \in \mathcal{E}$ we have $G\left(-\widetilde{M}\right) \leq \epsilon/2$ and $G\left(\widetilde{M}\right) \geq 1 - \epsilon/2$. Now put $M = \widetilde{M} + \alpha$, fix $F \in \mathcal{D}$ and choose $G \in \mathcal{E}$ in such a way that $\phi_{G,\alpha}(F) \leq \left(\chi_{rc}\right)_{\mathcal{A}_c}(\mathcal{D}) + \epsilon/2$. Then we have on the one hand
\begin{eqnarray*}
F(-M) &=& F\left(-\widetilde{M} - \alpha\right)\\ &\leq& G\left(- \widetilde{M}\right) + (\left(\chi_{rc}\right)_{\mathcal{A}_c}(\mathcal{D}) + \epsilon/2)\\ &\leq& \left(\chi_{rc}\right)_{\mathcal{A}_c}(\mathcal{D}) + \epsilon,
\end{eqnarray*}
and on the other
\begin{eqnarray*}
F(M) &=& F\left(\widetilde{M} + \alpha\right)\\ &\geq& G\left(\widetilde{M}\right) - (\left(\chi_{rc}\right)_{\mathcal{A}_c}(\mathcal{D}) + \epsilon/2)\\ &\geq& 1 - (\left(\chi_{rc}\right)_{\mathcal{A}_c}(\mathcal{D}) + \epsilon),
\end{eqnarray*}
entailing that $\chi_e(\mathcal{D}) \leq \left(\chi_{rc}\right)_{\mathcal{A}_c}\left(\mathcal{D}\right)$.

\end{proof}

\begin{gev}[Classical Prokhorov's Theorem]
For $\mathcal{D} \subset \mathcal{F}$ the following are equivalent.
\begin{itemize}
	\item[(1)] The collection $\mathcal{D}$ is weakly relatively sequentially compact.
	\item[(2)] The collection $\mathcal{D}$ is tight.
\end{itemize}
\end{gev}

\end{document}